\documentclass[12pt]{amsart}
\usepackage{amsmath}
\usepackage{amssymb}
\usepackage{geometry}
\usepackage{graphicx}
\usepackage{amsthm}
\usepackage{setspace}
\usepackage[final]{showkeys}

\usepackage{mathrsfs}        
\usepackage{bbm}   
 
\def\bna{B_{\epsilon,n}(x)}
\def\bnb{\widetilde{B}_{\epsilon,n}(x)}
\def\bnc{\widetilde{\partial B}_{\epsilon,n}(x)}
\def\tb{\widetilde{B}}
\def\tbna{\tau_{\bna}}
\def\tbnb{\tau_{\bnb}}

\allowdisplaybreaks

\newcommand{\const}{{\text{const.}}}

\usepackage{epstopdf}
\usepackage{amsthm}
\usepackage{enumerate}

\usepackage{dsfont}

\usepackage{showlabels}

\newtheorem{thm}{Theorem}
\newtheorem{lemma}{Lemma}

\providecommand{\abs}[1]{\lvert \, #1 \, \rvert}

\title[Entry times distribution on metric spaces]{Entry times distribution for dynamical 
balls on metric spaces}
  \date{\today}

\begin{document}

\begin{abstract}
We show that the entry and return times for dynamic balls (Bowen balls) is exponential
for systems that have an $\alpha$-mixing invariant measure with certain regularities. We also show that systems modeled by Young's tower has exponential hitting time distribution for dynamical balls
\end{abstract}

\author{N Haydn}
\thanks{N Haydn, Department of Mathematics, University of Southern California,
Los Angeles, 90089-2532. E-mail: {\tt {nhaydn@usc.edu}}.} 
\author{F Yang}
\thanks{F Yang, Department of Mathematics, University of Southern California,
Los Angeles, 90089-2532. E-mail: {\tt {yang617@usc.edu}}.}

\maketitle

\section{Introduction}

In this paper the distribution of entry and return times are studied for continuous
maps on metric spaces. There are a great many results on the distribution
of return times to cylinder sets for instance from 1990 by Pitskel~\cite{Pit}
(see also~\cite{Den}) for Axiom A
maps using Markov partitions and also by Hirata. For $\psi$-mixing maps,
Galves and Schmitt~\cite{GS} have introduced a method to obtain the
 limiting distribution for the first entry time to cylinder sets. Abadi then
 generalised that approach to $\phi$ and also $\alpha$-mixing measures.
 The nature of the return set however is critical for the longtime statistics
 of return and Lacroix and Kupsa~\cite{L02,KL} have given examples where for an 
 ergodic system any return time distribution can be realised by taking 
 a limit along a suitably chosen sequence of return sets.
 For entry and returns to balls, Pitskel's result shows using an approximation
 argument that for Axiom A maps on the two-dimensional torus return times
 are Poisson distributed. Recently, Chazottes and Collet~\cite{CC13} showed a similar
 result for attractors in the case of exponentially decaying correlations.
 This was in~\cite{HW} extended to polynomial decay of correlations
 where the error terms are logarithmic. A similar result without error terms
 and requiring sufficient regularity of the invariant measure was proven in~\cite{PS}
 (see also~\cite{H13}).
 In this paper we consider another class of entry sets, namely dynamic balls
 on metric spaces. It has been shown elsewhere that dynamic balls exhibit 
 good limiting statistic and in particular have the equipartition property
 which for partitions is associated with the theorem of Shannan-McMillan-Breiman~\cite{Man}.
 Similarly, a theorem of Ornstein and Weiss~\cite{OW,OW2} has a counterpart for 
 metric balls~\cite{Var}:
 It was shown that the exponential growth rate of the re-entry times is equal to the 
 metric entropy.
 
 In this paper we first study the distribution of the first entry time assuming that
 there is a generating partition which is $\alpha$-mixing. This requirement is 
 satisfied by a large number of systems, in particular by those which allow
 Young's tower construction. This is shown in section~\ref{markov.towers}.
 In section~3 we deduce the first return time distribution.

\section{Main results}
Let $(X, d)$ be a compact metric space, $T:X\circlearrowleft$ and $(X, T , \mu)$ be a continuous 
transformation and $\mu$ a $T$-invariant Borel probability measure on $X$.   For a set $B\subset X$ 
denote by $\tau_B (y)$ the first time when the orbit  of $y$ enters $B$, i.e.
 $\tau_B(y) = \min\{j>0: T^jy \in B\}\in\mathbb{N}\cup\{\infty\}$. The first return time $\tau_B|_B$ is almost surely finite 
 by Poincar\'e's recurrence theorem and  satisfies $\int_B\tau_B\,d\mu=1$ by Kac's theorem 
 if $\mu$ is ergodic and $\mu(B)>0$. A large number of results on the limiting distribution have been
 proven in the case when $B$ are cylinder sets, most notably by Galves and Schmitt~\cite{GS}
 for $\psi$-mixing measures where they introduced a method which later was in particular by 
 Abadi~\cite{Abadi01,Abadi04} extended to the first entries and returns for $\phi$-mixing and
  $\alpha$-mixing measures.

If $T$ is a map on a metric space $\Omega$ 
with metric $d$, then the $n$th Bowen ball is given by 
$B_{\varepsilon,n}(x)=\{y\in\Omega: d(T^jx,T^jy)<\varepsilon, 0\le j<n\}$. 
Bowen balls are used to define the metric entropy and also the pressure 
for potentials
(see e.g.~\cite{Wal}). 
Then in~\cite{BK} the equivalent for the theorem of Shannon-McMillan-Breiman
was proven for Bowen balls. Namely
$$
h(\mu)=\lim_{\epsilon\to 0}\lim_{n\to\infty} \frac 1n\abs{\log\mu(B_{\epsilon,n}(x))}
$$
 for almost every $x$. 
In~\cite{Var} (see also~\cite{DW}) that for an ergodic $T$-invariant probability measure $\mu$
one has
$$
\lim_{\varepsilon\to0}\lim_{n\to\infty}\frac1n\log R_{\varepsilon,n}(x)=h(\mu)
$$
almost everywhere, where $R_{\varepsilon,n}(x)=\tau_{B_{\varepsilon,n}(x)}(x)$ is 
the recurrence time to the Bowen ball (the limit in $n$ is $\limsup$ or $\liminf$).
In this paper we want to address the distribution of the entry time function
$\tau_{B_{\varepsilon,n}(x)}$ and of the higher order returns $\tau_{B_{\varepsilon,n}(x)}^j$
which are defined by $\tau_{B_{\varepsilon,n}(x)}^1=\tau_{B_{\varepsilon,n}(x)}$
and recursively
 $\tau_{B_{\varepsilon,n}(x)}^{j+1}=\tau_{B_{\varepsilon,n}(x)}
 +\tau_{B_{\varepsilon,n}(x)}^j\circ T^{\tau_{B_{\varepsilon,n}(x)}}$.
 
 In order to get meaningful results we have to assume some mixing property.
 We will require the measure to be $\alpha$-mixing with respect to a finite or countably
 infinite generating partition.
 For that purpose let $\mathcal{A}$ be a measurable partition of $X$ which is generating and 
 either finite or 
 countably infinite. Denote by $\mathcal{A}^k=\bigvee_{j=0}^{k-1}T^{-j}\mathcal{A}$ its
 $n$-th join and write $\gamma_n=\mbox{diam}(\mathcal{A}^n)$ for its diameter. 
 Clearly $\gamma_n$ is a decreasing sequence.
 We shall require that the measure $\mu$ is $\alpha$-mixing with respect to a such
 a partition $\mathcal{A}$, that is 
$$
|\mu(A \cap T^{-n-k} B) - \mu(A)\mu(B)| \le \alpha(k)
$$
for all $A \in \sigma(\mathcal{A}^n)$, $B \in \sigma(\bigcup_j \mathcal{A}^j )$,
where $\alpha(k)$ is a decreasing function which converges to zero as $k\to\infty$.

In addition to a mixing property we will require some regularity of the measure. 
For $0 < \delta < \epsilon$  define the function
$$
\varphi (\epsilon , \delta, x) = \frac{\mu(B(x, \epsilon + \delta)) - \mu(B(x, \epsilon - \delta))}{\mu(B(x, \epsilon  ))} 
$$
for $x\in X$. The function $\varphi$ measures the proportion  of the measure of the annulus  
$B(x, \epsilon  + \delta) \setminus B(x, \epsilon - \delta)$ to the ball $B(x, \epsilon )$. 
This is needed in order to control the approximation of balls by cylinder sets below.

 For $\epsilon>0 $ and $ n \in \mathbb{N}$ we then define the $(\epsilon,n)$-Bowen ball 
 $$
 \bna = \{y: \sup\limits_{0 \le k \le n-1} d(T^kx, T^ky) < \epsilon \}.
 $$ 
 We now can formulate our main results.


\begin{thm}\label{thm2}  Let $\mu$ an $\alpha$-mixing $T$-invariant probability measure on
 $\Omega$.
 Assume that there exist constants $0<\gamma<1$ and $\zeta, \kappa>0$ such that $ \mbox{diam}(\mathcal{A}^n) =\mathcal{O}(\gamma^n)$, $\alpha(n)=\mathcal{O}(n^{-(2+\kappa)})$ and   
$$
\varphi(\epsilon, \delta, x) \le \frac{C_{\epsilon}}{|\log \delta|^{3+\zeta}}
$$
for some constant $C_{\epsilon}>0$ that does not depend on $x$ and all $\delta$ small enough.

Then there exists an $\omega>0$ and a constant $C_1$ so that
$$
\big| \mathbb{P}(\tau_{\bna}  > \frac{t}{\lambda_{\bna}\mu(\bna)} ) - e^{-t} \big| 
\le C_1\mu(\bna)^\omega.
$$ 
\end{thm}

For simplicity let us put $B=\bna$. If, as in the next result, the measure has 
good regularity then we relax the condition on the decrease of the diameter of 
cylinders considerably.

\begin{thm} \label{thm3}
Assume that there exist constants, $a,\kappa,\zeta>0$ satisfying $a\zeta>3$, such that 
$ \mbox{diam}(\mathcal{A}^n)=\mathcal{O}(n^{-a})$, $\alpha(n) =\mathcal{O}(n^{-(2+\kappa)})$ and 
$$
\varphi(\epsilon, \delta, x) \le C_{\epsilon} \delta^{\zeta}
$$
for some constant $C_\epsilon$.

Then there exists an $\omega>0$ and a constant $C_2$ so that 
$$
 \big| \mathbb{P}(\tau_{\bna}  > \frac{t}{\lambda_{\bna}\mu(\bna)} ) - e^{-t} \big|
\le C_2\mu(\bna)^\omega.
$$ 
\end{thm}

While the previous two theorems give us limiting results for the entry times distribution,
the next two theorems establish equivalent results for the return times.

For all set $A\subset\Omega$ define the {\em period} of $A$ by
 $\tau(A) = \min\{k>0: T^{-k}A \cap A \ne \emptyset\}$ and put for any $\Delta<1/\mu(A)$
$$
a_{A} = \mathbb{P}_{A}(\tau_{A}> \tau(A) + \Delta).
$$
In our setting we will choose $N(n) \le \Delta \le 1/\mu(\bnb) $, where $N(n)$ will be determined
later. Again we write for simplicity $\tb = \bnb$ and $B = \bna$.

\begin{thm}\label{B_return} Let $\mu$ an $\alpha$-mixing.
Assume that there is a $\zeta>0$ so that 
$$
\varphi(\epsilon, \delta, x) \le \frac{C_{\epsilon}}{|\log \delta|^{5+\zeta}}
$$r
for some constant $C_\epsilon$. The remaining conditions are as in Theorem~\ref{thm2}.

Then there exists an $\omega>0$ so that 
$$
\left| \mathbb{P}_B(\tau_B > \frac{t}{\lambda_B \mu(B)}) - a_Be^{-t} \right|
\le C_3\mu(B)^\omega
$$
for some constant $C_3$ and a parameter $\lambda_B$ which is bounded as in
Lemma~\ref{l4}.
\end{thm}

\begin{thm}\label{B_return2}
Assume  there are $a,\kappa,\zeta>0$ satisfying $a\zeta>5$, such that 
$ \mbox{diam}(\mathcal{A}^n)=\mathcal{O}(n^{-a})$, $\alpha(n) =\mathcal{O}(n^{-(2+\kappa)})$ and 
$$
\varphi(\epsilon, \delta, x) \le C_{\epsilon}\delta^\zeta.
$$
The remaining conditions are as in Theorem~\ref{thm3}.

Then there exists an $\omega>0$ so that 
$$
\left| \mathbb{P}_B(\tau_B > \frac{t}{\lambda_B \mu(B)}) - a_Be^{-t} \right|
\le C_4\mu(B)^\omega
$$
for some constant $C_4$.
\end{thm}

 The next result will be our principal technical result on which all the other theorems
 are based. For that purpose let $N(n)$ be an increasing sequence. 
 We  want to approximate the Bowen balls$\bna$ by a unions
  on $N(n)$-cylinders from the inside. For this purpose put 
$$
\bnb= \bigcup\limits_{A^{N(n)} \in \mathcal{A}^{N(n)}, A^{N(n)} \subset \bna } A^{N(n)}
$$ 
which is the largest  union of all $N(n)$-cylinders contained in $\bna$. The following is
 our main result.

\begin{thm}\label{mainthm} Let $\mu$ be an $\alpha$-mixing $T$-invariant probability measure
on $\Omega$. Assume there exist $ \epsilon_0>0$ and an increasing sequence 
$\{N(n)\}_{n=1}^{\infty}$ satisfying $n<N(n) < \frac{1}{4}\mu(\bna)^{-1}$ such that
\begin{equation}{\label{qn1}}
 \varphi(\epsilon , \gamma_{N(n)-k} , T^k x) \le\vartheta_n(\epsilon) \cdot \frac{\mu(\bna)}{ns}
\end{equation}
for all $\epsilon < \epsilon_0$, $x \in X$, $0\le k\le n-1$, where $s = \alpha^{-1}(C'\mu(\tb))+N(n)$
 for some $0<C'<1$ and $\vartheta_n(\epsilon)\to0$ as $n\to\infty$ for every $\epsilon$.

 Then there exist  $\lambda_{\bna}$   with $\frac{C}{s} < \lambda_{\bna}< 2$ and constants 
 $C_5,C_6$ such that   
\begin{eqnarray*}
  \left| \mathbb{P}(\tau_{\bna}  > \frac{t}{\lambda_{\bna}\mu(\bna)} ) - e^{-t} \right| \hspace{-4cm}&&\\
  &\le &\vartheta_n(\epsilon)\frac{t}{s\lambda_{\bna}}+ 2f\mu(\bna) + C_5 \frac{sN(n)}{f} + C_6s\frac{\alpha(N(n))}{f\mu(\tb)}
\end{eqnarray*}
 for all $f\in\left(2N(n),\frac12\mu(\bna)^{-1}\right)$.
\end{thm}

\section{First hitting time distribution for Bowen balls}
In this section we will prove Theorems~\ref{thm2} and~\ref{thm3},
but  first we state several lemmata, the first one of which is evident.

\begin{lemma}\label{l1} For all $n$ and $x$ we have $\sum_{k=1}^{N(n)} \mu(\bnb \cap T^{-k}\bnb) \le N(n)\mu(\bnb) $ and $\mathbb{P}(\tbnb \le t) \le t\mu(\bnb)$.
\end{lemma}

To simply notation, we fix $\epsilon$ and $n$ for a moment  and  write $B = \bna$ and  $\tb = \bnb$.

\begin{lemma}\label{l2} For all $\Delta,f$ such that $f \ge \Delta > N(n)$ and $g \in \mathbb{N}$ we have
$$
\left| \mathbb{P}(\tau_{\tb} > f+g) - \mathbb{P}(\tau_{\tb} > g) \mathbb{P}(\tau_{\tb} > f)  \right| \le 2 \Delta\mu(\bnb) + \alpha(\Delta - N(n))
$$
\end{lemma}

\begin{proof} We proceed in the traditional way splitting the difference into three parts:
\begin{eqnarray*}
\big| \mathbb{P}(\tau_{\tb} > g+f) - \mathbb{P}(\tau_{\tb} > g) \mathbb{P}(\tau_{\tb} > f)  \big|\hspace{-4cm}&&\\
&\le & \big| \mathbb{P}(\tau_{\tb} > g+f) - \mathbb{P}(\tau_{\tb} > g \cap \tau_{\tb} \circ T^{g+ \Delta} > f- \Delta ) \big|\\
&& + \big| \mathbb{P}(\tau_{\tb} > g \cap \tau_{\tb} \circ T^{g+ \Delta} > f- \Delta )  - \mathbb{P}(\tau_{\tb} > g) \mathbb{P}(\tau_{\tb} > f- \Delta)  \big|\\
&& +   \big| \mathbb{P}(\tau_{\tb} > g) \mathbb{P}(\tau_{\tb} > f- \Delta)  - \mathbb{P}(\tau_{\tb} > g) \mathbb{P}(\tau_{\tb} > f)  \big|\\
&=& I + II + III.
\end{eqnarray*}
The first term is estimated as follows
$$
I = \mathbb{P}(\tau_{\tb} > g \cap \tau_{\tb} \circ T^{g+ \Delta} > f- \Delta \cap \tau_{\tb} \circ T^g \le \Delta)
\le  \mathbb{P}(\tau_{\tb} \le \Delta)
\le  \Delta\mu(\tb).
$$
Similarly for the third term
$$
III = \mathbb{P}(\tau_{\tb} > g) \mathbb{P}(f-\Delta<\tau_{\tb}\le f)
\le  \mathbb{P}(\tau_{\tb} > g) \Delta\mu(\tb) \le \Delta\mu(\tb).
$$
For the second term we use the $\alpha$-mixing property to obtain
$$
II=\big| \mathbb{P}(\tau_{\tb} > g \cap \tau_{\tb} \circ T^{g+ \Delta} > f- \Delta )  - \mathbb{P}(\tau_{\tb} > g) \mathbb{P}(\tau_{\tb} > f- \Delta)  \big| \le \alpha(\Delta - N(n)).
$$
The three parts combined now prove the lemma.
\end{proof}

Let us now put   $\theta =\theta(f)= -\log\mathbb{P}(\tau_{\tb}>f)$ where $f>0$. We then
have the following estimate.

\begin{lemma}\label{l3} Let $f>\Delta > N(n)$ then for all $k\ge1$  we have
$$
\big| \mathbb{P}(\tau_{\tb}> kf) - e^{-\theta k} \big| \le \frac{2 \Delta \mu(\tb) + \alpha(\Delta - N(n))}{\mathbb{P}(\tau_{\tb} \le f )}.
$$
\end{lemma}

\begin{proof}
Clearly the lemma hold for $k=1$ by definition of $\theta$. For  $k>1 $ we use induction.
For the induction step we obtain:
\begin{eqnarray*}
\big| \mathbb{P}(\tau_{\tb}> (k+1)f) - e^{-\theta (k+1)} \big|\hspace{-4cm}&& \\
&\le & | \mathbb{P}(\tau_{\tb}> (k+1)f) -  \mathbb{P}(\tau_{\tb}> kf)\cdot e^{-\theta}| + | \mathbb{P}(\tau_{\tb}> kf)\cdot e^{-\theta} - e^{-\theta (k+1)} |\\
&\le & 2\Delta \mu (\tb) + \alpha(\Delta - N(n)) + e^{-\theta}\big| \mathbb{P}(\tau_{\tb}> kf) - e^{-\theta k} \big|\\
&\le & 2\Delta \mu (\tb) + \alpha(\Delta - N(n)) + e^{-\theta} (2 \Delta \mu(\tb)+\alpha(\Delta -N(n))) \cdot(1+ e^{-\theta}+ \cdots + e^{-\theta (k-2)})\\
&= &  (2 \Delta \mu(\tb)+\alpha(\Delta -N(n))) \cdot(1+ e^{-\theta}+ \cdots + e^{-\theta (k-1)}).
\end{eqnarray*}
Hence 
$$
\big| \mathbb{P}(\tau_{\tb}> kf) - e^{-\theta k} \big| \le (2 \Delta \mu(\tb)+\alpha(\Delta -N(n))) 
\frac{1}{1-e^{-\theta}}
$$
for all $k\in\mathbb{N}$ and the lemma follows since 
$\frac{1}{1 - e^{-\theta}} = \frac{1}{\mathbb{P}(\tau_{\tb} \le f)}$.
\end{proof}

For subsets $B\subset X$ let us define
$$
\lambda_{B,f}  = \frac{-\log \mathbb{P}(\tau_B > f )}{f\mu(B)}.
$$ 
For the approximations $\bnb$ we then obtain the following estimate.

\begin{lemma}\label{l4} Let $f \in \mathbb{N}$ be such that  $f \mu(\bnb) \le  \frac{1}{2}$. 
Then  there exist $C_7 > 0$ such that
$$ 
\frac{C_7}{s} \le \lambda_{\bnb,f} \le 2
$$
where, as before, $s = \alpha^{-1}(C'\mu(\bnb))+N(n)$ for some $0<C'<1$
\end{lemma}

\begin{proof}
We follow the  proof in Galves-Schmitt \cite{GS} and Abadi \cite{Abadi06}. To estimate $\lambda_{\bnb,f}$ we use the simple estimate
$$
\frac{\theta}{2} \le 1-e^{-\theta} \le \theta
$$
for all $\theta \in [0, 1]$. Let us write $\tb$ for $\bnb$ and note that
 $\mathbb{P}\{\tau_{\tb} \le f\} = 1 - e^{-\theta}$ as $\theta = - \log \mathbb{P}\{\tau_{\tb} > f\}$.
  By Lemma~\ref{l1},
$$
\lambda_{\tb,f} = \frac{\theta}{f\mu(\tb)} \le \frac{2\mathbb{P}\{\tau_{\tb} \le f\}}{f\mu(\tb)} \le 2 
$$
For the lower bound, notice that 
$\{\tau_{\tb} >f\} = \bigcap_{j=0}^{[f]} \Big(T^{-js}(\tb)\Big)^c
\subset \bigcap_{j=0}^{[\frac{f}{s}]} \Big(T^{-js}(\tb)\Big)^c$. As in the proof of Lemma~\ref{l3}
we obtain by induction
$$
\mu\left(\bigcap_{j=0}^{k+1} \Big(T^{-js+1}(\tb)\Big)^c\right)
\le \mu\left(\bigcap_{j=0}^{k} \Big(T^{-js+1}(\tb)\Big)^c\right)\mu(\tb^c)+\alpha(s-N(n))
$$
which yields
$$
\mu\left(\bigcap_{j=0}^{[\frac{f}s]} \Big(T^{-js+1}(\tb)\Big)^c\right)
\le\mu(\tb^c)^{[\frac{f}s]}+\alpha(s-N(n))\frac1{1-\mu(\tb^c)}.
$$
Consequently
$$
\mathbb{P}(\tau_{\tb}>f) \le (1-\mu(\tb))^{f/s}+\alpha(s-N(n))\frac{1-(1-\mu(\tb))^{f/s}}{\mu(\tb)},
$$
and therefore
\begin{align*}
\mathbb{P}(\tau_{\tb}\le f) \ge& (1-(1-\mu(\tb))^{f/s})\big(1-\frac{\alpha(s-N(n))}{\mu(\tb)}\big)\\
\ge&\frac{f}{s}\mu(\tb)\big(1-\frac{\alpha(s-N(n))}{\mu(\tb)}\big).
\end{align*}
Thus
$$
\lambda_{\tb,f} = \frac{\theta}{f\mu(\tb)} \ge \frac{\mathbb{P}\{\tau_{\tb} \le f\}}{f \mu(\tb)}
\ge  \frac{1}{s}\big(1-\frac{\alpha(s-N(n))}{\mu(\tb)}\big).
$$
In particular since $s = \alpha^{-1}(C'\mu(\tb))+N(n)$ we get $\lambda_{\tb,f}  \ge  \frac{C_7}{s}$ and $\mathbb{P}(\tau_{\tb} \le f) \ge  \frac{C_7f\mu(B)}{s}$ for some constant $C_7$.
\end{proof}

Now let us define the `annulus'

$$
\bnc= \bigcup\limits_{A^{N(n)} \in \mathcal{A}^{N(n)}, A^{N(n)} \cap \partial\bna \ne \emptyset}   A^{N(n)}.
$$
We then have that $\bna \setminus \bnb \subset\bnc$.
We also have $\tbna \ge \tbnb$ since  $\bnb \subset \bna$.
The following lemma estimates the size of the annulus.

\begin{lemma}\label{l5}
With the notation as above (and in particular with $s = \alpha^{-1}(C'\mu(\tb))+N(n)$) we obtain
$$
\mu(\bnc)=  \vartheta_n(\epsilon) \frac{\mu(\bna)}{s}.
$$
\end{lemma}
\begin{proof}
Since $T$ is continuous,  $\partial \bna \subset \bigcup\limits_{k=0}^{n-1}T^{-k}\partial B(T^kx, \epsilon)$. Hence if $A^{N(n)} \cap \partial\bna \ne \emptyset$ then we must have $A^{N(n)-k}(T^ky) \cap \partial B(T^kx,\epsilon) \ne \emptyset$ for some $0 \le k \le n-1$, $y \in A^{N(n)}$. Notice that $\mbox{diam}(A^{N(n)-k}(T^ky)) \le \gamma_{N(n)-k}$, we have
\begin{align*}
\bnc \subset & \bigcup\limits_{k=0}^{n-1} T^{-k}(B(\partial B(T^kx, \epsilon), \gamma_{N(n)-k}))\\
\subset& \bigcup T^{-k}(B(T^kx, \epsilon+ \gamma_{N(n)-k}) \setminus B(T^kx, \epsilon - \gamma_{N(n)-k})),
\end{align*}
hence
\begin{align*}
\mu(\bnc) \le& n \cdot\sup_{0 \le k \le n-1} \mu(B(T^kx, \epsilon+ \gamma_{N(n)-k}) \setminus B(T^kx, \epsilon - \gamma_{N(n)-k}))\\
=& n \cdot \sup_{0 \le k \le n-1} \{\varphi(\epsilon, \gamma_{N(n)-k}, T^kx)\cdot \mu(B(T^kx, \epsilon))\}\\
=& \vartheta_n(\epsilon)  \frac{\mu(\bna)}{s}  \sup_{0 \le k \le n-1}\mu(B(T^kx, \epsilon))\\
=& \vartheta_n(\epsilon) \frac{\mu(\bna)}{s}.
\end{align*}
In particular we have $\mu(\bna)/\mu(\bnb) = \mathcal{O}(1)$.
\end{proof}

\begin{lemma}  For $f\le \frac{1}{2}\mu(\bnb)^{-1}$ one has 
$$
\Big| \mathbb{P}(\tbna > \frac{t}{\lambda_B \mu(\bna)}) -  \mathbb{P}(\tbnb > \frac{t}{\lambda_B \mu(\bnb)}) \Big|
\le3\frac{\vartheta_n(\epsilon) t}{s\lambda_{\tb,f}} 
$$ 
for all $t > 0$.
\end{lemma}

\begin{proof} Let us write  $\lambda_B = \lambda_{\tb,f}$. Then
\begin{align*}
&\big| \mathbb{P}(\tau_B > \frac{t}{\lambda_B \mu(B)}) - \mathbb{P}(\tau_{\tb}> \frac{t}{\lambda_B \mu(\tb)}) \big| \\
\le & \Big| \mathbb{P}(\tau_B > \frac{t}{\lambda_B \mu(B)}) - \mathbb{P}(\tau_{B}> \frac{t}{\lambda_B \mu(\tb)}) \Big|+\Big| \mathbb{P}(\tau_B > \frac{t}{\lambda_B \mu(\tb)}) - \mathbb{P}(\tau_{\tb}> \frac{t}{\lambda_B \mu(\tb)}) \Big|\\
=& I + II.
\end{align*}
We estimate the two terms separately.

For the term~$I$ first notice that $\bnb \subset \bna$ which implis 
$\tbnb \ge \tbna$ and $\frac{t}{\lambda_B\mu(B)}\le\frac{t}{\lambda_B\mu(\tb)} $. Therefore
\begin{align*}
I &\le \mathbb{P}\left(\frac{t}{\lambda_B\mu(B)}\le \tau_B\le\frac{t}{\lambda_B\mu(\tb)}\right)\\
& \le \frac{t}{\lambda_B}\left( \frac{1}{\mu(\tb)} - \frac{1}{\mu(B)}\right)\mu(B) \\
&\le 2\frac{\vartheta_n(\epsilon)t}{s\lambda_B}
\end{align*}
as $\frac{\mu(B)}{\mu(\tb)}\le2$ and where we used that 
\begin{equation}\label{r.diff}
\frac{1}{\mu(\tb)} - \frac{1}{\mu(B)}
 =\frac{\mu(B \setminus \tb)}{\mu(B)\mu(\tb)}
\le \frac{\mu(\bnc)}{\mu(B)\mu(\tb)}
\le\vartheta_n(\epsilon)\frac1{s\mu(\tb)}
\end{equation}
by Lemma~\ref{l5}.

The term~$II$ we estimate as follows:
\begin{align*}
\text{II} &= \mathbb{P}\left(\left\{\tau_B \le \frac{t}{\lambda_B\mu(\tb)}\right\} \cap 
\left\{\tau_{\tb} >\frac{t}{\lambda_B\mu(\tb)}\right\} \right) \\
& \le \mathbb{P}\left(\tau_{\bnc} \le \frac{t}{\lambda_B\mu(\tb)}\right)\\
& \le \frac{t}{\lambda_B}\frac{\mu(\bnc)}{\mu(\tb)}\\
&=  \frac{\vartheta_n(\epsilon)t}{s\lambda_B},
\end{align*}
where in the last line we proceeded as for the term~$I$ above. Since by Lemma~\ref{l4}, 
$s\lambda_B>C_7/2$  the result follows.
\end{proof}


\begin{proof}[Proof of Theorem~\ref{mainthm}] We have to estimate $ |\mathbb{P}(\tbnb > \frac{t}{\lambda_B \mu(\bnb)}) - e^{-t}| $, where, as before, $\lambda_B = \lambda_{\tb,f}$.
We put $\Delta = 2N(n) $ and pick $f >\Delta= 2N(n)$ with $f \le \frac{1}{2}\mu(\bnb)^{-1}$.
Then $t>0$ can be written as $t = kf + r $ with $0 \le r < f$ and $k$ integer. Set $t' = t-r = kf$, then
\begin{eqnarray*}
|\mathbb{P}(\tau_{\tb} > t) - e^{ - \lambda_B \mu(\tb)t}|\hspace{-2cm}&&\\
&\le & |\mathbb{P}(\tau_{\tb} > t) -\mathbb{P}(\tau_{\tb} > t')| + |\mathbb{P}(\tau_{\tb} > t') - e^{ - \lambda_B \mu(\tb)t'}| + |e^{ - \lambda_B \mu(\tb)t'} - e^{ - \lambda_B \mu(\tb)t}|\\
&=& I + II + III.
\end{eqnarray*}
The first term is easily estimated by
$$
I = \mathbb{P}(t'<\tau_{\tb} \le t) \le r\mu(\tb) <f\mu(\tb).
$$
For the third term we use the mean value theorem according to which there exist 
$t_0 \in [\lambda_B \mu(\tb)t',\lambda_B \mu(\tb)t]$ such that
$$
III = e^{-t_0}\lambda_B  \mu(\tb) r \le 2 f \mu(\tb)
$$                          
using Lemma~\ref{l4} in the last estimate.

To the second term, $II$, we apply Lemma~\ref{l3} and obtain
\begin{align*}
II \le&  \frac{2 \Delta \mu(\tb) + \alpha(\Delta - N(n))}{\mathbb{P}(\tau_{\tb} \le f )}\\
\le& \frac{s(2 \Delta \mu(\tb) + \alpha(\Delta - N(n)))}{Cf\mu(\tb)}\\
=& C_5 \frac{sN(n)}{f} + C_6s\frac{\alpha(N(n))}{f\mu(\tb)}.
\end{align*}

All three estimates combined yield
$$
 |\mathbb{P}(\tbnb > \frac{t}{\lambda_B \mu(\bnb)}) - e^{-t}|  \le 2f\mu(\bna) + C_5 \frac{sN(n)}{f} + C_6s\frac{\alpha(N(n))}{f\mu(\tb)}.
$$
\end{proof}

For the remaining results we will use approximations of Bowen balls by unions of cylinder
sets of lengths $N(n)$. For that puropose let us establish the following notation.
For some $0<\eta < \frac{1}{2}, \beta\in(\eta,1)$, which will be determined later, we define 
$N(n) = \mu(\bna)^{-\eta} $ (length of cylinders), $f = \mu(B)^{-\beta}$, $\bnb = \bigcup\limits_{D \in \mathcal{A}^{N(n)}, D \subset \bna} D$ (inner approximation) and 
$$
\lambda_{\bna}  = \frac{-\log \mathbb{P}(\tau_{\bnb} > f )}{f\mu(\bnb)}.
$$

\begin{proof}[Proof of Theorem~\ref{thm2}.] In order to apply Theorem~\ref{mainthm} we first verify ~(1) with 
$\gamma_n = \mathcal{O}(\gamma^n)$. Fix some $0<\eta<\frac{1}{2}$ and set 
$N(n) =[ \mu(\bna)^{-\eta}]$. Then
$$
 \varphi (\epsilon , \gamma_{N(n)-k} , T^kx) \le \frac{C_{\epsilon}}{|\log \gamma^{N(n) - k}|^{3+ \zeta}}
  \le \frac{C_{\epsilon}}{N(n)^{3+\zeta}}
  =C_{\epsilon}\mu(B)^{(3+\zeta)\eta}.
 $$
Since  
\begin{equation}\label{ss}
s =  \alpha^{-1}(C'\mu(\tb))+N(n) \le c_1 \mu(B)^{-\frac{1}{2+\kappa}} + \mu(B)^{-\eta} 
\end{equation}
for some constant $c_1$, we have
$$
\frac{ns\varphi(\epsilon , \gamma_{N(n)-k} , T^kx)}{\mu(B)} \le\vartheta_n(\epsilon),
$$
where
$$
\vartheta_n(\epsilon)
 \le C_{\epsilon}ns\mu(B)^{(3+\zeta)\eta-1}
\le c_2n \mu(B)^{(3+\zeta)\eta-1}
\left(\mu(B)^{-\frac{1}{2+\kappa}} + \mu(B)^{-\eta}\right),
$$
which converges to $0$ if $\eta >\eta_0= \max\{\frac{1}{2+\zeta}, \frac{1}{3+\zeta}\frac{3+\kappa}{2+\kappa}\}$.

Applying Theorem~\ref{mainthm} yields as $s\lambda_{\bna}\ge C_7$:
$$
  \big| \mathbb{P}(\tau_{\bna}  > \frac{t}{\lambda_{\bna}\mu(\bna)} ) - e^{-t} \big|
   \le \vartheta_n(\epsilon)t+ 2f\mu(\bna) + C_5 \frac{sN(n)}{f} + C_6s\frac{\alpha(N(n))}{f\mu(\tb)}.
$$
Since $f = \mu(B)^{-\beta}$ for some $\beta<1$,  the second term on the RHS converges to $0$.
 The last two terms then are bounded as follows:
$$
C_5 \frac{sN(n)}{f} + C_6s\frac{\alpha(N(n))}{f\mu(\tb)}
 \le c_3\mu(B)^\beta\left(\mu(B)^{-\eta}+ \mu(B)^{\eta(2+\kappa)-1}\right)\left(\mu(B)^{-\frac{1}{2+\kappa}} + \mu(B)^{-\eta}\right).
$$
In order that all terms converge to 0 we need $\eta>\eta_0$ so that
$$
\omega_1=\beta - \max\left\{ \eta,\frac{1}{2+\kappa}\right\}-\max\left\{ \eta, 1- \eta(2+\kappa)\right\}
$$
is positive.
This can be achieved by picking $\eta\in(\eta_1,\frac12)$, where 
$\eta_1=\max\{\frac1{2+\zeta},\frac1{2+\kappa}\}$
is less than $\frac12$ (note $\eta_1\ge\eta_0$). This also implies that 
$\omega_2=(3+\zeta)\eta-\max\{\frac1{2+\kappa},\eta\}$ is positive.
Now put $\omega=\min\{\omega_1,\omega_2,1-\beta\}$.
\end{proof}

\begin{proof}[Proof of Theorem~\ref{thm3}.] In order to apply Theorem~\ref{mainthm} 
we verify that condition~\eqref{qn1} holds:
\begin{align*}
 \varphi(\epsilon , \gamma_{N(n)-n} , T^kx) &\le C_{\epsilon} \gamma_{N(n)-k} ^{\zeta}\\
 & \le c_1 (N(n)-k)^{-a\zeta}\\
 & \le c_2 N(n)^{-a\zeta}
\end{align*}
(for some $c_1, c_2$) and therefore
$ns\varphi(\epsilon , \gamma_{N(n)-k} , T^kx)/\mu(B) \le\vartheta_n(\epsilon)$, where
$$
\vartheta_n(\epsilon) \le c_3ns\mu(B)^{a\zeta\eta-1}
\le c_4n\mu(B)^{a\zeta\eta-1}\left(\mu(B)^{-\frac{1}{2+\kappa}} + \mu(B)^{-\eta}\right).
$$
The last expression converges (exponentially fast) to zero if $a\zeta\eta-1-\max\{\frac1{2+\kappa},\eta\}$
is positive which can be achieved by picking $\eta<\frac12$ close enough to $\frac12$.

The remainder of the proof is identical to the proof of  Theorem~\ref{thm2}.
\end{proof}

 \section{First return time distribution}
 In this section we will prove Theorems~\ref{B_return} and~\ref{B_return2} which
 establish the limiting distribution of the first return time to Bowen balls and provide
 rates of convergence.
 We use the same notation as in the previous section.

$\tau(A)$ Be the period of $A$ and as in~\cite{Abadi06} denote by 
$a_{A} = \mathbb{P}_{A}(\tau_{A}> \tau(A) + \Delta)$ the relative size of the 
set of long returns, where  $\Delta<1/\mu(A)$.
Again we put $\tb = \bnb$ and $B = \bna$ and let us first prove the following to
lemmata.

We shall need the following result which compares the return times distributions
for the dynamical ball $B$ and the approximated set $\tilde{B}$.

\begin{lemma}\label{l8}
There exists a constant $C_8$ so that 
$$
\left| \mathbb{P}_B(\tau_{B} > \frac{t}{\lambda_B\mu(B)}) - \mathbb{P}_{\tb}(\tau_{\tb} > \frac{t}{\lambda_B \mu(\tb)}) \right| \le C_8\frac{t\vartheta_n(\epsilon)}{\mu(B)}.
$$
\end{lemma}
\begin{proof}
Let us first estimate the following term:
$$
I=\left|\mathbb{P}_{\tb}(\tau_{\tb} > \frac{t}{\lambda_B\mu(\tb)}) 
-\mathbb{P}_{B}(\tau_{\tb} > \frac{t}{\lambda_B\mu(\tb)})\right|
$$
which is split into two parts $I\le I_1+I_2$. For for the first part we obtain
by Lemma~\ref{l5}
$$
I_1=\frac1{\mu(B)}\left|\mathbb{P}(\{\tau_{\tb} > \frac{t}{\lambda_B\mu(\tb)}\}\cap\tb) 
-\mathbb{P}(\{\tau_{\tb} > \frac{t}{\lambda_B\mu(\tb)}\}\cap B)\right|
\le\frac{\mu(\bnc)}{\mu(B)}
\le\frac{\vartheta_n(\epsilon)}s.
$$
The second part is by~\eqref{r.diff}
$$
I_2
=\mathbb{P}\!\left(\left\{\tau_{\tb} > \frac{t}{\lambda_B\mu(\tb)}\right\}\cap B\right)\left|\frac1{\mu(\tb)}-\frac1{\mu(B)}\right|
\le\frac{\vartheta_n(\epsilon)}{s\mu(\tb)}.
$$
Hence 
$$
I\le2\frac{\vartheta_n(\epsilon)}{s\mu(\tb)}.
$$

Let us next estimate the term
$$
II=\Big| \mathbb{P}_B(\tau_{B} > \frac{t}{\lambda_B\mu(B)}) - \mathbb{P}_B(\tau_{\tb} > \frac{t}{\lambda_B \mu(\tb)}) \Big|,
$$
which again splits into two parts $II=II_1+II_2$ as follows. The first part is
\begin{eqnarray*}
 II_1
 &=&\left| \mathbb{P}_B(\tau_{B} > \frac{t}{\lambda_B\mu(B)}) 
 - \mathbb{P}_B(\tau_{B} > \frac{t}{\lambda_B \mu(\tb)}) \right|\\
 &\le& \mathbb{P}_B(\frac{t}{\lambda_B \mu(B)} < t < \frac{t}{\lambda_B \mu(\tb)}) \\
 &\le& \frac{1}{\mu(B)}\Big( \frac{t}{\lambda_B\mu(\tb)} - \frac{t}{\lambda_B\mu(B)}\Big)\mu(B) \\
 &\le&\vartheta_n(\epsilon)\frac{t}{s\lambda_{B}\mu(\tb)}
\end{eqnarray*}
by~\eqref{r.diff}. For the second part we obtain
$$
II_2
=\left| \mathbb{P}_B(\tau_{B} > \frac{t}{\lambda_B\mu(\tb)}) - \mathbb{P}_B(\tau_{\tb} > \frac{t}{\lambda_B \mu(\tb)}) \right|
\le \mathbb{P}_B(\tau_{B \setminus \tb} < \frac{t}{\lambda_B \mu(\tb)})
\le \frac{t}{\lambda_B}\frac{\mu(B \setminus \tb)}{\mu(B) \mu(\tb)}
$$
and therefore by Lemma~\ref{l5}
$$
II_2\le\frac{t}{\lambda_B}\frac{\mu(\bnc)}{\mu(B)\mu(\tb) } 
\le\frac{t\vartheta_n(\epsilon)}{s\lambda_B\mu(\tb)}.
$$

Finally we obtain for some constant $C_8$ that 
$$
I+II_1+II_2\le C_8\frac{t\vartheta_n(\epsilon)}{\mu(B)}
$$
where we used that $s\lambda_B\ge C_7$ by Lemma~\ref{l4} and $\frac{\mu(B)}{\mu(\tb)}=\mathcal{O}(1)$.
\end{proof}

The following lemma is similar to Lemma~\ref{l2}.

\begin{lemma}\label{l2condition} For all $\Delta,f,g$ such that $f \ge \Delta > N(n)$, $g \ge \Delta + \tau(\tb)$ we have
$$
\big| \mathbb{P}_{\tb}(\tau_{\tb} > f+g) - \mathbb{P}_{\tb}(\tau_{\tb} > g) \mathbb{P}(\tau_{\tb} > f)  \big| \le 2 \Delta\mu(\bnb) + 2\frac{\alpha(\Delta - N(n))}{\mu(\tb)}
$$
\end{lemma}

\begin{proof}
We proceed as in the proof of Lemma~\ref{l2} to write the left-hand-side as\\ I + II + III.  The only difference is in I:
\begin{align*}
I = & |\mathbb{P}_{\tb}(\tau_{\tb} > f+g) - \mathbb{P}_{\tb}(\tau_{\tb}>g \cap \tau_{\tb} \circ T^{g + \Delta} > f - \Delta)| \\
\le & \mathbb{P}_{\tb}(\tau_{\tb}\circ T^g \le \Delta)\\
= & \frac{1}{\mu(\tb)} \mathbb{P}(\tb\cap \{\tau_{\tb}\circ T^g \le \Delta\})\\
\le & \mathbb{P}(\tau_{\tb} \le \Delta) + \frac{\alpha(\Delta - N(n))}{\mu(\tb)} \le \Delta \mu(\tb) + \frac{\alpha(\Delta - N(n))}{\mu(\tb)} 
\end{align*}
by the $\alpha$-mixing property. The estimates of the terms II and III are identical to the
proof of Lemma~\ref{l2}.
\end{proof}

\begin{proof}[Proof of Theorem~\ref{B_return}]
We will first show that $\mathbb{P}_{\tb}(\tau_{\tb} > \frac{t}{\lambda_B \mu(\tb)})$ satisfies the exponential law. 

For all $t > (\tau(\tb) + 2\Delta)\lambda_B \mu(\tb)$, let $u = \frac{t}{\lambda_B \mu{\tb}}$ we have
\begin{eqnarray*}
\big|\mathbb{P}_{\tb}(\tau_{\tb} > u) - \mathbb{P}_{\tb}(\tau_{\tb} > \tau(\tb) + \Delta) e^{-t}\big|\hspace{-6cm}&&\\
&\le & \big| \mathbb{P}_{\tb}(\tau_{\tb} >  u) - \mathbb{P}_{\tb} (\tau_{\tb} > \tau(\tb) + \Delta) \mathbb{P}(\tau_{\tb} > u - (\tau(\tb) + \Delta))  \big| \\ 
&& + \big| \mathbb{P}_{\tb} (\tau_{\tb} > \tau(\tb) + \Delta) \mathbb{P}(\tau_{\tb} > u - (\tau(\tb) + \Delta)) -   \mathbb{P}_{\tb}(\tau_{\tb} > \tau(\tb) + \Delta) e^{-t}      \big|\\
&= & I + II
\end{eqnarray*}
where
\begin{align*}
I \le 2\Delta \mu(\tb) + 2\frac{\alpha(\Delta - N(n))}{\mu(\tb)}
\end{align*}
by Lemma~\ref{l2condition}. For $II$ we have

\begin{align*}
II = & a_B \big| \mathbb{P}(\tau_{\tb} >  \frac{t}{\lambda_B \mu(\tb)} - (\tau(\tb) + \Delta)) - e^{-t} \big|\\
\le & a_B\big|  \mathbb{P}(\tau_{\tb} >  \frac{t}{\lambda_B \mu(\tb)} - (\tau(\tb) + \Delta)) -  e^{-t + (\tau(\tb) + \Delta)\lambda_B \mu(\tb)} \big|+ a_B \big|e^{-t} - e^{-t + (\tau(\tb) + \Delta)\lambda_B \mu(\tb)}\big|.
\end{align*}

To the first term we apply Theorem~\ref{thm2} with the parameter value $t'=t-(\tau(\tb) + \Delta)\lambda_B \mu(\tb)$
and to the second term we apply the  Mean Value Theorem. Hence 
$II\le c_1\mu(\tb)^{\omega_1}$ for some $\omega_1>0$ from Theorem~\ref{thm2}. 
This proves Theorem~\ref{B_return} for the set $\tb$. To prove the theorem for the set $B$ we use
Lemma~\ref{l8} and put $N(n) = \mu(B)^{-\eta}$ for some $\eta \in (0, 1/2)$. Thus
\begin{eqnarray*}
 \Big| \mathbb{P}_B(\tau_{B} > \frac{t}{\lambda_B\mu(B)}) 
 - \mathbb{P}_{\tb}(\tau_{\tb} > \frac{t}{\lambda_B \mu(\tb)}) \Big|\hspace{-3cm}&& \\
&\le &  C_8 \frac{t\vartheta_n(\epsilon)}{\mu(B)}\\
&= & \frac{\ \mathcal{O}(t) ns}{|\log\gamma^N|^{5 + \zeta}\mu(B)^2}\\
&= &  \mathcal{O}(t) ns\mu(B)^{(5 + \zeta)\eta -2 }\\
&=&  \mathcal{O}(t)n\mu(B)^{(5 + \zeta)\eta -2 }\left(\mu(B)^{-\frac{1}{2+\kappa}}+\mu(B)^{-\eta}\right).
\end{eqnarray*}
by~\eqref{ss} and as  $s\lambda_B =\mathcal{O}(1)$ and $s =  \alpha^{-1}(C'\mu(\tb))+N(n) $ by Lemma~\ref{l4}.
 A choice of $\eta$ close to $\frac12$ will achieve that  $\omega_2=(5+\zeta)\eta-2-\max\{\frac1{2+\kappa},\eta\}$
 is positive. Now put $\omega=\min\{\omega_1,\omega_2\}$.
 \end{proof}

\begin{proof}[Proof of Theorem~\ref{B_return2}.]
The first part of the proof is identical to the  proof of Theorem~\ref{B_return}. For the 
second part we get different estimates as $\mbox{diam}(\mathcal{A}^n)=\mathcal{O}(n^{-a})$
for some $a$. To prove the theorem for the set $B$ we use
Lemma~\ref{l8} and put $N(n) = \mu(B)^{-\eta}$ for some $\eta \in (0, 1/2)$. Thus
\begin{eqnarray*}
 \Big| \mathbb{P}_B(\tau_{B} > \frac{t}{\lambda_B\mu(B)}) 
 - \mathbb{P}_{\tb}(\tau_{\tb} > \frac{t}{\lambda_B \mu(\tb)}) \Big|\hspace{-3cm}&& \\
&\le&  C_8 \frac{t\vartheta_n(\epsilon)}{\mu(B)}\\
&= & \frac{\ \mathcal{O}(t) ns}{N(n)^{a\zeta}\mu(B)^2}\\
&=&  \mathcal{O}(t)n\mu(B)^{a\zeta\eta -2 }\left(\mu(B)^{-\frac{1}{2+\kappa}}+\mu(B)^{-\eta}\right).
\end{eqnarray*}
by~\eqref{ss}.
 A choice of $\eta$ close to $\frac12$ will achieve that  $\omega_2=a\zeta\eta-2-\max\{\frac1{2+\kappa},\eta\}$
 is positive. Now put $\omega=\min\{\omega_1,\omega_2\}$ ($\omega_1$ from the proof of
 Theorem~\ref{B_return}).
\end{proof}

 \section{Maps with decaying correlations}\label{markov.towers}
 
 In this section we show how the results of the previous section can be applied to 
 dynamical systems that can be modelled by a Markov tower as Young constructed in~\cite{Y2,Y3}.

We assume that  $T$ is a differentiable map on the manifold $X$. Then one assumes 
there is a subset $\Omega_0\subset X$ with the following properties:\\
(i) $\Omega_0$ is partitioned into disjoint sets $\Omega_{0,i}, i=1,2,\dots$ and
there is a {\em return time function} $R:\Omega_0\rightarrow\mathbb{N}$,  constant on 
the partition elements $\Omega_{0,i}$,  such that $T^R$ maps $\Omega_{0,i}$ bijectively 
to the entire set $\Omega_0$. We write $R_i=R|_{\Omega_{0,i}}$. 
Moreover, it is assumed that the $\Omega_{0,i}$ are 
rectangles, that is, if $\gamma^u(x)$ denotes the unstable 
leaf through $x\in\Omega_{0,i}$ and $\gamma^s(y)$ the stable leaf at $y\in\Omega_{0,i}$,
then there is a unique interestion $\gamma^u(x)\cap\gamma^s(y)$ which also lies 
in $\Omega_{0,i}$. It is also assumed that the $\Omega_{0,i}$ satisfy the 
Markov property. 
If $\gamma^u$ and $\hat\gamma^u$ are two unstable leaves (in some $\Omega_{i,0}$),
then the holonomy $\Theta:\gamma^u\to\hat\gamma^u$ is given by 
$\Theta(x)=\hat\gamma^u\cap\gamma^s(x)$, $x\in\gamma^u$. \\
(ii) For $j=0,1,\dots, R_i-1$ put $\Omega_{j,i}=\{(x,j): x\in\Omega_{0,i}\}$
and define $\Omega =\bigcup_{i=1}^\infty\bigcup_{j=0}^{R_i-1}\Omega_{j,i}$.
$\Omega$ is the {\em Markov tower} for the map $T$. It has the associated partition 
$\mathcal{A}=\{\Omega_{j,i}:\; 0\le j<R_i, i=1,2,\dots\}$ which typically is countably infinite.
The map $F: \Omega\to\Omega$ is given by is given by
  $$\
  F(x,j)=\left\{\begin{array}{ll}(x,j+1) &\mbox{if } j<R_i-1\\ 
  (\hat{F}x,0)&\mbox{if } j=R_i-1\end{array}\right.
  $$
  where we put $\hat{F}=F^R$ for the induced map on $\Omega_0$.\\
 (iii) The {\em separation function} $s(x,y)$,  $x,y\in\Omega_0$, is defined as the largest positive
  $n$ so that $(T^R)^jx$ and  $(T^R)^jy$ lie in the same sub-partition elements 
 for $0\le j<n$. That is $(T^R)^jx, (T^R)^jy\in\Omega_{0,i_j}$ 
 for some $i_0,i_1,\dots,n-1$. We extend the separation function to all of $\Omega$
 by putting $s(x,y)=s(T^{R-j}x,T^{R-j}y)$ for $s,y\in\Omega_{j,i}$.\\
   (iv) Let $\nu$ be a finite given `reference' measure on $\Omega$ and let 
  $\nu_{\gamma^u}$ be the conditional measure on the unstable leaves.
  We assume that the Jacobian $JF=\frac{d(F^{-1}_*\nu_{\gamma^u})}{d\nu_{\gamma^u}}$ 
  is  H\"older continuous in the following sense:
  There exists a $\gamma\in(0,1)$ so that 
 $$
 \left|\frac{JF^Rx}{JF^Ry}-1\right|\le\mbox{\rm const} \gamma^{s(\hat{F}x,\hat{F}y)}
 $$
 for all $x,y\in \Omega_{0,i}$, $i=1,2,\dots$.

 If the return time $R$ is integrable with respect to  $\nu$ then by~\cite{Y3} 
 Theorem~1 there exists an $F$-invariant probability measure
 $\mu$ (SRB measure) on $\Omega$ which is absolutely continuous with respect to $\nu$.

\begin{thm}\label{thm.young}
 Let $\mu$ be the SRB measure for a differentiable map $T$ on a manifold $X$.
 If the return times decay at least polynomially with power $\lambda>5+\sqrt{15}$, then
 $$
 \lim_{\epsilon\to0}\lim_{n\to\infty}\mathbb{P}\left(\tau_{B_{\epsilon,n}(x)}\ge \frac t{\lambda_{\bna}\mu(B_{\epsilon,n}(x))}\right)= e^{-t}
 $$
 for $t>0$ and almost every $x\in X$.
 \end{thm}
 
 For each $n\in\mathbb{N}$ the elements of the $n$th join 
 $\mathcal{A}^n=\bigvee_{i=0}^{n-1}T^{-i}\mathcal{A}$ 
of the partition $\mathcal{A}=\{\Omega_{i,j}\}$ are called $n$-cylinders and form a new partition
 of $\Omega$, a refinement of the original partition. The
$\sigma$-algebra  generated by all $n$-cylinders $\mathcal{A}^\ell$, for all $\ell\ge1$, is 
the $\sigma$-algebra of the system $(\Omega,\mu)$.

 \begin{lemma}\label{markov.alpha.mixing}
 The invariant measure $\mu$ is $\alpha$-mixing with respect to the partition $\mathcal{A}$,
 where $\alpha(k)\sim p(k)$.
 \end{lemma}
 
 \begin{proof} Denote by $\mathcal{C}_\gamma$ the space of H\"older continuous functions $\varphi$ 
 on $\Omega$ for which $|\varphi(x)-\varphi(y)|\le C_\varphi \gamma^{s(x,y)}$. 
 If $C_\varphi$ is smallest then  $\|\varphi\|_\gamma=|\varphi|_\infty+C_\varphi$ defines
 a norm and $\mathcal{C}_\gamma=\{\varphi: \;\|\varphi\|_\gamma<\infty\}$.
 Let  $\mathcal{L}:\mathcal{C}_\gamma\rightarrow\mathcal{C}_\gamma$ be the transfer
 operator defined by  $\mathcal{L}\varphi(x)=\sum_{x'\in T^{-1}x}\frac{\varphi(x')}{JT(x')}$, $\varphi\in\mathcal{C}_\gamma$. Then $\nu$ is a fixed point of its adjoint, 
 i.e.\ $\mathcal{L}^*\nu=\nu$ and  
 $h=\frac{d\mu}{d\nu}=\lim_{n\rightarrow\infty}\mathcal{L}^n\lambda$ is H\"older continuous, 
 where $\lambda$ can be any initial density distribution in $\mathcal{C}_\gamma$.
In fact, by~\cite{Y3} Theorem~2(II) one has
\begin{equation}\label{convergence}
\|\mathcal{L}^k\lambda-h\|_{\mathscr{L}^1}\le p(k) \|\lambda\|_\gamma
\end{equation}
where the `decay function' $p(k)=\mathcal{O}(k^{-\beta})$ if the return times decay polynomially 
with power $\beta$, that is if $\nu(R> j)\le \const j^{-\beta}$.
If the return times decay exponentially, i.e.\ if $\nu(R>j)\le\const \vartheta^j$ for some $\vartheta\in(0,1)$, then there is a $\tilde\vartheta\in(0,1)$ so that $p(k)\le \const\tilde\vartheta^k$.
 
As in the proof of~\cite{Y3} Theorem~3 we put $\lambda=\mathcal{L}^nh\chi_A$ which is a strictly
positive function. Then $\eta=\frac\lambda{\mu(A)}$ is a density function as 
$\nu(\lambda)=\nu(\mathcal{L}^nh\chi_A)=\nu(h\chi_A)=\mu(A)$. 
Since by~\cite{HP}
there exists a constant $c_1$ so that $\|\mathcal{L}^n\chi_A\|_\gamma\le c_1$ for all 
$A\in\sigma(\mathcal{A}^n)$ and $n$ we see that $\|\lambda\|_\gamma\le c_1$ 
uniformly in $n$ and $A\in\sigma(\mathcal{A}^n)$. 
Hence
\begin{eqnarray*}
\mu(A\cap T^{-k-n}B)-\mu(A)\mu(B)&=&
\nu(h \chi_A(\chi_B\circ T^{k+n}))-\nu(h\chi_A)\nu(h\chi_B)\\
&=&\mu(A)(\nu( \chi_B\mathcal{L}^k\eta)-\nu(h\chi_B))\\
&=&\mu(A)\int\chi_B(\mathcal{L}^k\eta-h)\,d\nu\\
&=&\int_B(\mathcal{L}^k\lambda-\mu(A)h)\,d\nu.\label{young.mixing1}
\end{eqnarray*}
Using the estimates from the $\mathscr{L}^1$-convergence of $\mathcal{L}^k\eta-h$ from~(\ref{convergence}) yields
\begin{eqnarray*}
\left|\mu(A\cap T^{-k-n}B)-\mu(A)\mu(B)\right|
&\le&\mu(A)\int\chi_B|\mathcal{L}^k\lambda-h|\,d\nu\\
&\le&\mu(A)c_1\|\eta\|_\gamma p(k)\\
&\le&c_3p(k)\label{young.mixing2}
\end{eqnarray*}
as $\|\eta\|_\gamma=\frac1{\mu(A)}\|\lambda\|_\gamma\le\frac{C_3}{\mu(A)}$. 
In particular we can write 
$$
\left|\mu(A\cap T^{-k-n}B)-\mu(A)\mu(B)\right| \le \alpha(k)
$$
for all $A\in\sigma(\mathcal{A}^n)$, $B\in\sigma(\bigcup_{j\ge1}\mathcal{A}^j)$,
where $\alpha(k)=c_3p(k)$.
\end{proof}

\begin{lemma} \label{annulus.towers}Let $\xi<\frac\lambda2-1$.
Then there exists an $\epsilon_0$ so that for every $\delta<\epsilon_0$ there exists a set
 $\mathcal{U}_\delta\subset X$, of measure $\mathcal{O}(\abs{\log\delta}^{-\xi})$ 
so that  $\varphi(\epsilon,\delta,x)=\mathcal{O}(\abs{\log\delta}^{-\xi})$ uniformly in $x\not\in\mathcal{U}_\delta$.
\end{lemma}

\begin{proof}
It was shown in~\cite{HW} Proposition~6.1 that for all $w$ large enough there exists a set
 $\mathcal{U}\subset X$ such that $\mu(\mathcal{U})=\mathcal{O}((w\abs{\log\epsilon})^{-\xi})$ 
 and $\varphi(\epsilon,\epsilon^w,x)=\mathcal{O}((w\abs{\log\epsilon})^{-\xi})$ uniformly in 
 $x\not\in\mathcal{U}$ where $\xi$ is any number less than $\frac\lambda2-1$. 
 Hence there exists an $\epsilon_0>0$ so that we can write $\delta=\epsilon^w$ with 
 $w$ large enough (larger than $\frac2u(D+1)-1$ where $D$ is the dimension of the 
 manifold $X$ and $u$ is the dimension of the unstable leaves) for all $\delta<\epsilon_0$.
 Since $\log\delta=w\log\epsilon$ we obtain the statement of the lemma.
\end{proof}

Let us denote by $\tilde\Omega_{j,i}$ the principal parts of $\Omega_{j,i}$. For
integers $N,m$ ($N>\!\!>m$) we put
$\tilde\Omega_{j,i}=\{x\in\Omega_{j,i}:R(\hat{F}^jx)\le s\;\forall\;j=0,\dots,[N/m]\}$.
In this way we pick out the return times that are not too long. In particular 
$\tilde\Omega_{0,i}=\varnothing$ if $R_i>m$. Let us put 
$\tilde\Omega=\bigcup_i\bigcup_{j=0}^{R_i-1}\tilde\Omega_{j,i}$ (disjoint unions).

We also define $\omega(m)=\sqrt{\sum_{i:R_i>m}R_i\,\nu(\Omega_{0,i})}$ and note that
 $\omega(m)=\mathcal{O}(m^{-\frac{\lambda-1}2})$.

\begin{lemma}\cite{HW} \label{tall.towers}
There exists a constant $C_9$ and for $N,n,m\ge1$ ($N>n,m$) there exist sets $\mathcal{V}_{N,m}\subset M$ 
such that the non-principal part contributions are estimated as
$$
\mu(\mathcal{B} \cap (\Omega \setminus \tilde{\Omega}))
 <\sqrt{n+2}\,\omega(m)\mu(B_{\epsilon,n})
$$
for any $\mathcal{B}\subset B_{\epsilon,n}(x)$ and  $x\not\in\mathcal{V}_{N,m}$  where
$$
\mu(\mathcal{V}_{N,m}) \leq C_9 \sqrt{n+2}\,\omega(m).
$$
\end{lemma}

\begin{proof}[Proof of Theorem~\ref{thm.young}]
In order to apply Theorem~\ref{mainthm} to $B_{\epsilon,n}(x)\cap\tilde\Omega$
we will pick $\eta\in(0,\frac{1}{2})$ below and put $N(n) =[ \mu(\bna)^{-\eta}]$. 
We then choose $m=N^\alpha$ for some $\alpha\in(\frac1{\lambda-1},1)$ (see estimate of $\mathcal{F}$ below). 
Then according to Lemma~\ref{tall.towers} $\mbox{diam}(A)\le \gamma^\frac{N}m$
for some $\gamma<1$ for all $n$-cylinders $A$ which belong to $ \tilde\Omega$. As in the 
proof of Theorem~\ref{thm2} we then conclude that
$$
 \varphi (\epsilon , \gamma_{N(n)-k} , T^kx)
  \le \frac{c_1 m^\xi}{N(n)^{\xi}}
  \le c_1 m^\xi\mu(B)^{\xi\eta}
$$
(for some constant $c_1$) provided $T^kx\not\in\mathcal{U}_{\gamma_{N(n)-k}}$ for $k=0,\dots,n-1$.
Since $s = \alpha^{-1}(C' \mu(B)) + N(n)\le c_2\mu(B)^{-\frac{1}{\lambda}} + \mu(B)^{-\eta}$, we obtain
$$
\frac{ns \cdot \varphi(\epsilon , \gamma_{N(n)-k} , T^kx)}{\mu(B)} 
 \le c_3n \cdot\mu(B)^{\eta\xi-1-\eta\alpha\xi} \left(\mu(B)^{-\frac{1}{\lambda}} + \mu(B)^{-\eta}\right).
$$
The RHS converges to zero if $\eta\xi-1-\eta\alpha\xi-\max\{\frac1\lambda,\eta\}$ is positive.
To satisfy Lemma~\ref{annulus.towers} it is required that $\xi<\frac\lambda2-1$. Then can choose
 $\alpha\in(\frac1{\lambda-1},1)$ in such a way that the above expression is positive for an $\eta<\frac12$
 close to $\frac12$.
 This can be done if $\lambda>5+\sqrt{15}$.

We now proceed as in the proof Theorem~\ref{thm2} to estimate the contribution to the error made
by  $(B\setminus\tb)\cap\tilde\Omega$. For the portion that lies in $\Omega\setminus\tilde\Omega$
we use Lemma~\ref{tall.towers} and thus obtain combining the two contributions:
$$
  \big| \mathbb{P}(\tau_{\bna}  > \frac{t}{\lambda_{\bna}\mu(\bna)} ) - e^{-t} \big|
   \le c_4\left( t \mu(B)^a +  \mu(B)^b+ \sqrt{N}\,\omega(m)\mu(B)\right)
$$
provided $x$ does not lie in the forbidden set 
$\mathcal{F}=\mathcal{V}_{N,m}\cup\bigcup_{k=0}^{n-1}T^{-k}\mathcal{U}_{\gamma_{N(n)-k}}$ whose 
measure is by Lemmata~\ref{tall.towers} and~\ref{annulus.towers}
bounded by 
$$
\mu(\mathcal{F})\le c_5\left( \sqrt{N+2}\,\omega(m)+n\abs{\log\gamma_{N(n)}}^{-\xi}\right)
\le c_6\left(\sqrt{N}\,m^{-\frac{\lambda-1}2}+nN^{-\xi}m^\xi\right)
$$
which goes to zero as $n\to\infty$ since $m=N^\alpha$ and $\frac1{\lambda-1}<\alpha<1$.
Thus $\mu(\mathcal{F})\to0$ as $n\to\infty$
and therefore
$\mathbb{P}(\tau_{\bna}  > \frac{t}{\lambda_{\bna}\mu(\bna)} ) \longrightarrow e^{-t}$
as $n\to\infty$, $\epsilon\to0$ for every $x\not\in\liminf_{n\to\infty,\epsilon\to0}\mathcal{F}_{\epsilon,n}$.

\end{proof}

\end{document}